\documentclass[CRMATH,Unicode,manuscript]{cedram}

\usepackage{bm}
\usepackage{booktabs}
\usepackage{amssymb}
\usepackage[normalem]{ulem}

\newcommand{\distribution}{f}
\newcommand{\latticevelocity}{\lambda}

\newcommand{\momentum}{m}
\newcommand{\average}[1]{\overline{#1}}
\newcommand{\vectorial}[1]{\bm{#1}}
\newcommand{\operatorial}[1]{\bm{#1}}
\newcommand{\levelletter}{\ell}
\newcommand{\maxlevel}{\levelletter_{\text{max}}}
\newcommand{\minlevel}{\levelletter_{\text{min}}}
\newcommand{\indexletter}{k}
\newcommand{\populationindex}{\alpha}
\newcommand{\velocityletter}{e}
\newcommand{\normalizedvelocityletter}{c}
\newcommand{\superscript}[2]{#1^{#2}}
\newcommand{\subscript}[2]{#1_{#2}}
\newcommand{\collided}{\star}
\newcommand{\cellletter}{C}
\newcommand{\timestep}{\Delta t}
\newcommand{\spacestep}{\Delta x}
\newcommand{\leveldifference}{\Delta \levelletter}
\newcommand{\adaptiveroundbrackets}[1]{\left ( #1 \right )}

\newcommand{\velocitynumber}{q}
\usepackage{scalerel,stackengine}
\stackMath
\newcommand\reallywidehat[1]{%
\savestack{\tmpbox}{\stretchto{%
  \scaleto{%
    \scalerel*[\widthof{\ensuremath{#1}}]{\kern-.6pt\bigwedge\kern-.6pt}%
    {\rule[-\textheight/2]{1ex}{\textheight}}
  }{\textheight}%
}{0.5ex}}%
\stackon[1pt]{#1}{\tmpbox}%
}
\newcommand\reallywidedoublehat[1]{%
\savestack{\tmpbox}{\stretchto{%
  \scaleto{%
    \scalerel*[\widthof{\ensuremath{#1}}]{\kern-.6pt\bigwedge\kern-.6pt}%
    {\rule[-\textheight/2]{1ex}{\textheight}}
  }{\textheight}%
}{0.5ex}}%
\stackon[-0.3mm]{\stackon[1pt]{#1}{\tmpbox}}{\tmpbox}%
}

\newcommand{\reconstructed}[3]{\reallywidedoublehat{#1} \subscript{\superscript{\vphantom{#1}}{#2}}{#3}}
\newcommand{\reconstructionweight}{C}
\newcommand{\levelleft}{\levelletter_{\text{left}}}
\newcommand{\levelright}{\levelletter_{\text{right}}}
\newcommand{\leveljump}{\levelletter_{\text{jump}}}
\newcommand{\bigO}[1]{\mathcal{O}(#1)}

\title{Does the multiresolution lattice Boltzmann method allow to deal with waves passing through mesh jumps?}

\author{\firstname{Thomas} \lastname{Bellotti}\CDRorcid{0000-0002-4139-075X}}
\address{\label{address_X}CMAP, CNRS, Ecole polytechnique, Institut Polytechnique de Paris, 91128 Palaiseau Cedex, France.}
\email[T. Bellotti]{thomas.bellotti@polytechnique.edu}
\author{\firstname{Lo\"{\i}c} \lastname{Gouarin}}
\address{\label{address_X}CMAP, CNRS, Ecole polytechnique, Institut Polytechnique de Paris, 91128 Palaiseau Cedex, France.}
\email[L. Gouarin]{loic.gouarin@polytechnique.edu}
\author{\firstname{Benjamin} \lastname{Graille}\CDRorcid{0000-0002-6287-2627}}
\address{\label{address_ors}Université Paris-Saclay, CNRS, Laboratoire de mathématiques d’Orsay, 91405, Orsay, France.}
\email[B. Graille]{benjamin.graille@universite-paris-saclay.fr}
\author{\firstname{Marc} \lastname{Massot}\CDRorcid{0000-0001-8823-7667}}
\address{\label{address_X}CMAP, CNRS, Ecole polytechnique, Institut Polytechnique de Paris, 91128 Palaiseau Cedex, France.}
\email[M. Massot]{marc.massot@polytechnique.edu}

\thanks{The first author is supported by a PhD funding (year 2019) from the Ecole polytechnique.}


\begin{abstract} 
  We consider an adaptive multiresolution-based lattice Boltzmann scheme, which we have recently introduced and 
  studied from the perspective of the error control and the theory of the equivalent equations. This numerical strategy leads to high compression rates, error control and 
   its high accuracy has been explained on uniform and dynamically adaptive grids. However, one
   key issue with non-uniform meshes within the framework of lattice Boltzmann schemes is to properly handle  acoustic waves passing through a level jump of the grid. It usually yields spurious effects, in particular reflected waves. In this paper, we propose a simple mono-dimensional test-case for the linear wave equation with a fixed adapted mesh characterized by a potentially large level jump. We investigate this configuration with our original strategy and prove that we can handle and control the amplitude of    the reflected wave, which is of fourth order in the space step of the finest mesh. Numerical illustrations show that the proposed strategy outperforms the existing methods in the literature and allow to assess the ability of the method to handle the mesh jump properly.
\end{abstract}

\begin{document}
\maketitle

\section{Introduction}

In \cite{bellotti2021multiresolution1d, bellotti2021multidimensional2d} we have proposed a novel approach to perform time-dynamic grid adaptation and to devise lattice Boltzmann schemes to be deployed on such adaptive meshes using multiresolution analysis. This framework guarantees reduced memory footprints for problems with steep solution and potential gains in terms of computational time. 
This is obtained by storing the solution solely at the local level of refinement, which is essentially coarse in the case of problems involving shocks, see \cite{duarte2012new}.
Moreover, the preeminent feature of this strategy is that it yields a precise bound on the additional numerical error caused by the 
compression of the solution on a dynamically adapted grid 
coupled with the numerical scheme, thanks to the information on the local smoothness of the solution provided by multiresolution.
These peculiarities have been thoroughly investigated in these two works.
The next question was to clarify -- besides the previously elucidated error control -- the potential alteration of the physics of the simulated system by the adaptive method, which has been fully studied in \cite{bellotti2021lbmmreqeq}, concluding that our method of choice does not alter the reference lattice Boltzmann scheme up to local contributions of order four in the space step. This performance seems unprecedented in the literature.
This analysis has been carried out on uniform or locally uniform grids for smooth solutions using the equivalent equations \cite{dubois2008equivalent}. When non-smooth solutions are obtained, one must utilize the time-adaptive refinement to allow for the representation of singularities at the finest level of resolution.

The question which stimulated the present work has been raised by Pierre Lallemand and Fran\c{c}ois Dubois during a seminary at the IHP in Paris and is the following: how does the scheme introduced in our previous works \cite{bellotti2021multiresolution1d, bellotti2021multidimensional2d, bellotti2021lbmmreqeq} behave once an acoustic wave is forced to pass through grids of different resolutions, like the one on Fig.~ \ref{fig:configuration}? 
We emphasize that in the context of dynamic mesh adaptation, the only case in which waves could penetrate through a level jump is when they are locally very smooth.
For any numerical solver dealing with adapted meshes, beyond lattice Boltzmann schemes, it is known that a wave passing through a grid transition normally splits into two parts: the first one propagating through the interface and the second one which is reflected back.
The phenomenological reason for this is the different acoustic impedance of two media made up of grids at different levels of refinement.
Generally speaking, the aim is to devise numerical approaches which minimize the amplitude of the reflected waves.
This is particularly important in applications such as aeroacoustics \cite{gendre2017grid}, where spurious currents on the density field can have important consequences on the solution of the problem.

Many strategies, both for the representation of an adaptive grid and for the construction of adaptive lattice Boltzmann methods have been explored. For a review on the different techniques and the issues linked with grid transitions, the reader can consult \cite{lagrava2012revisiting} and \cite{horstmann2018hybrid}.
There is no widespread consensus on a standard test case to analyze the issue of reflected waves: we therefore consider a basic one-dimensional configuration which already introduces the main difficulties of more realistic multidimensional systems (\emph{e.g.} the D2Q9 scheme) to be found in the available literature, arising from the applications.
We also tested\footnote{The source code is available on \url{https://github.com/hpc-maths/samurai_cras_2021}.} our approach against the two-dimensional acoustic pulse test case from \cite{gendre2017grid, astoul2021lattice} yielding results -- not presented in this note -- fully compatible with the simpler analysis introduced here.


\section{Numerical method}

Let us briefly sketch the structure of our numerical method in order to make the paper self-contained.
The interested reader can consult \cite{bellotti2021multiresolution1d, bellotti2021multidimensional2d, bellotti2021lbmmreqeq} for more precision.

\subsection{Space and time discretization}

For the sake of presentation, we consider a one-dimensional setting on the bounded domain $\Omega = [0, 3]$. 
As in our previous work  \cite{bellotti2021lbmmreqeq} and references therein, we take cells $\cellletter_{\levelletter, \indexletter} = \left [2^{-\levelletter} \indexletter, 2^{-\levelletter}(\indexletter + 1) \right ]$ for levels of resolution $\levelletter \in \{ \minlevel, \dots, \maxlevel \}$ and $\indexletter \in \{0, \dots, 3 \times 2^{\leveldifference}-1  \}$, where $\leveldifference = \maxlevel - \levelletter$. The cell center is denoted by $x_{\levelletter, \indexletter}$.
Therefore, the space step of the grid at finest level of resolution is $\spacestep = 2^{-\maxlevel}$.
Cells with different levels $\levelletter$ can be employed to cover the domain $\Omega$ of interest as in Fig.~ \ref{fig:configuration}, which also illustrates the experimental configuration used in the sequel.

The time discretization is uniform with step $\timestep = \spacestep / \latticevelocity$, where $\latticevelocity > 0$ is the so-called ``lattice velocity''. The time-step is the same for all the levels $\levelletter \in \{ \minlevel, \dots, \maxlevel \}$ and determined by $\maxlevel$.
In many approaches available in the literature, level-wise time steps are employed. However, local time stepping would require a lot of attention and computational effort to preserve error control, thus preventing from a significant cost reduction (see \cite{bellotti2021multiresolution1d}, \cite{deiterding2016comparison} and \cite{lopes2019local}).

\begin{figure}
    \begin{center}
        \def\svgwidth{1.\textwidth}
        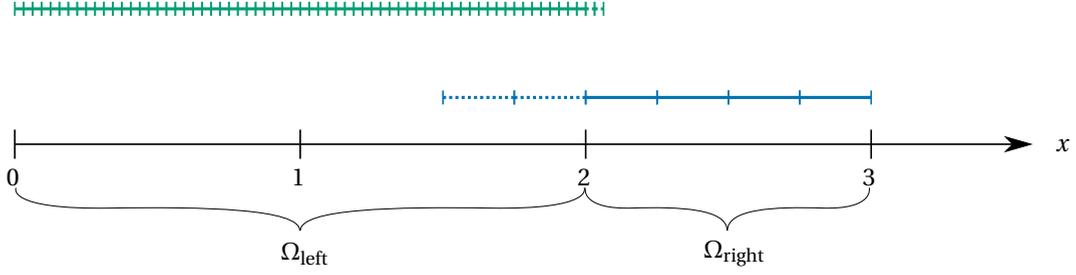
    \end{center}\caption{\label{fig:configuration}Example of domain $\Omega = \Omega_{\text{left}} \cup \Omega_{\text{right}}$ with $\Omega_{\text{left}} = [0, 2]$ finely meshed (green full line) and $\Omega_{\text{right}} = [2, 3]$ coarsely meshed (blue full line). Dashed lines represent ghost cells, where the solution needs to be updated to deploy the adaptive scheme.}
\end{figure}

\subsection{Lattice Boltzmann method}

The lattice Boltzmann method is a numerical algorithm based on  $\velocitynumber \in \mathbb{N^{\star}}$ velocities $({\velocityletter}_{\populationindex})_{\populationindex = 0}^{\populationindex = \velocitynumber - 1} \subset \latticevelocity \mathbb{Z}$ compatible with the lattice, with dimensionless counterparts ${\normalizedvelocityletter}_{\populationindex} := {\velocityletter}_{\populationindex}/\latticevelocity \in \mathbb{Z}$, $\populationindex = 0, \dots, \velocitynumber-1$.
The distribution function of the population moving with velocity ${\velocityletter}_{\populationindex}$ shall be denoted $\distribution^{\populationindex}$. 
The algorithm is made up of two phases at each time step.

\subsubsection{Collision phase}

For every cell $\cellletter_{\levelletter, \indexletter}$ belonging to the mesh, the collision phase is performed locally and is a diagonal relaxation in the space of the moments
\begin{equation*}
    \begin{cases}
        \average{\vectorial{\momentum}}_{\levelletter, {\indexletter}}(t) &= \operatorial{M} \average{\vectorial{\distribution}}_{\levelletter, {\indexletter}}(t), \nonumber \\
        \average{\vectorial{\distribution}}_{\levelletter, {\indexletter}}^{\collided} (t) &= \operatorial{M}^{-1} \adaptiveroundbrackets{(\operatorial{I} - \operatorial{S})  \average{\vectorial{\momentum}}_{\levelletter, {\indexletter}}(t) + \operatorial{S} \vectorial{\momentum}^{\text{eq}} (\average{\momentum}_{\levelletter, {\indexletter}}^{0}(t), \dots, \average{\momentum}_{\levelletter, {\indexletter}}^{\velocitynumber_c - 1}(t))},
    \end{cases}
\end{equation*}
where $\operatorial{M} \in \text{GL}(\velocitynumber, \mathbb{R})$, $\operatorial{S} \in \mathcal{M}_{\velocitynumber}(\mathbb{R})$ diagonal with $\text{rank}(\operatorial{S}) = \velocitynumber - \velocitynumber_{c}$ where the non-zero\footnote{Zero relaxation parameters are customarly taken for the conserved moments. However, different choices are present in the literature. They only influence the algorithm once introducing forcing terms \emph{via} the equilibria.} entries belong to $]0, 2]$. Here, $q_c$ is the number of conserved moments and the vector of moments at equilibrium $\vectorial{\momentum}^{\text{eq}}$ is a function of these conserved moments.

\subsubsection{Transport phase}

For every cell $\cellletter_{\levelletter, \indexletter}$ belonging to the mesh, the transport phase does not mix different populations. For every $\populationindex \in \{ 0, \dots, \velocitynumber-1 \}$ it comes under the form
\begin{equation}\label{eq:StreamPhase}
    \average{\distribution}_{\levelletter, \indexletter}^{\populationindex}(t + \timestep) = \average{\distribution}_{\levelletter, \indexletter}^{\populationindex, \star}(t) + \frac{1}{2^{\leveldifference}} \sum_{m = -2}^{2} \reconstructionweight_{\leveldifference, m}^{\populationindex} \average{\distribution}_{\levelletter, \indexletter + m}^{\populationindex, \star}(t),
\end{equation}
with weights recursively defined by
\begin{equation*}
    \begin{pmatrix}
        \reconstructionweight_{\leveldifference, -2}^{\populationindex} \\
        \reconstructionweight_{\leveldifference, -1}^{\populationindex} \\
        \reconstructionweight_{\leveldifference,  0}^{\populationindex} \\
        \reconstructionweight_{\leveldifference,  1}^{\populationindex} \\
        \reconstructionweight_{\leveldifference,  2}^{\populationindex}
    \end{pmatrix}
    = \begin{pmatrix}
        0 & -1/8 & 0 & 0 & 0 \\
        2 & 9/8 & 0 & -1/8 & 0 \\
        0 & 9/8 & 2 & 9/8 & 0 \\
        0 & -1/8 & 0 & 9/8 & 2 \\
        0 & 0 & 0 & -1/8 & 0
      \end{pmatrix}
    \begin{pmatrix}
        \reconstructionweight_{\leveldifference-1, -2}^{\populationindex} \\
        \reconstructionweight_{\leveldifference-1, -1}^{\populationindex} \\
        \reconstructionweight_{\leveldifference-1,  0}^{\populationindex} \\
        \reconstructionweight_{\leveldifference-1,  1}^{\populationindex} \\
        \reconstructionweight_{\leveldifference-1,  2}^{\populationindex}
    \end{pmatrix}, \quad \text{and} \quad
    \begin{cases}
        \reconstructionweight_{0, -\normalizedvelocityletter_{\populationindex}}^{\populationindex} &= 1, \\
        \reconstructionweight_{0, 0}^{\populationindex} &= -1, \\
        \reconstructionweight_{0, m}^{\populationindex} &= 0, \quad m \notin \{ 0, -\normalizedvelocityletter_{\populationindex} \}.
    \end{cases}
\end{equation*}
The transport phase \eqref{eq:StreamPhase} needs information stored on ghost cells (dashed in Fig.~ \ref{fig:configuration}), which thus must be updated. This is done by averaging for the ghost cells underneath more refined cells followed by the use of interpolations (from which \eqref{eq:StreamPhase} is recovered, see \cite{bellotti2021lbmmreqeq} for the details) for ghost cells above coarser cells.
It should be observed that the proposed method exactly conserves the moments which are conserved by the original lattice Boltzmann scheme, as remarked in \cite{bellotti2021multiresolution1d,bellotti2021multidimensional2d}.
Moreover, we stress that the aim of our numerical strategy is not to change the features of the original (on the uniform mesh) lattice Boltzmann scheme, with its strengths and weaknesses, but rather to minimize its perturbation once introducing non-uniform meshes.

\section{Target problem, results and discussion}

We introduce a simple system which sustains travelling acoustic waves which are eventually sent against a level jump in the computational grid.

\subsection{Equations and numerical scheme}

The target equation is the linear wave equation with velocity $c > 0$ (taken  $c = 1/2$ in the experiments, but any other value $c < 1$ would be fine) over the whole real line $\mathbb{R}$
\begin{equation*}
    \begin{cases}
        \partial_{tt} u - c^2 \partial_{xx} u = 0, \\
        u(t = 0, x) = u_0(x), \\
        \partial_t u(t = 0, x) = 0,
    \end{cases} \quad \Leftrightarrow \quad
    \begin{cases}
        \partial_{t} u + \partial_{x} v = 0, \\
        \partial_{t} v + c^2\partial_{x} u = 0, \\
        u(t = 0, x) = u_0(x), \\
        v(t = 0, x) = 0,
    \end{cases} 
\end{equation*}
which has been recast under the form of first order system for simulating it. 
For this test we simulate on the bounded domain $\Omega = \Omega_{\text{left}} \cup \Omega_{\text{right}}$ with $\Omega_{\text{left}} = [0, 2]$ and $\Omega_{\text{right}} = [2, 3]$ with the initial datum $u_0(x) = \text{exp}(-100(x-3/2)^2)$.
The most bare-bone lattice Boltzmann scheme to handle such equation -- yet yielding the difficulties of more sophisticated ones -- is a three velocity scheme $\velocitynumber = 3$ with two conserved moments ($u = m^0$ and $v = m^1$) 
\begin{equation*}
    \normalizedvelocityletter_{0} = 0, \quad \normalizedvelocityletter_{1} = 1, \quad \normalizedvelocityletter_{2} = -1, \quad
    \operatorial{M} = \left (
    \begin{matrix}
     1 & 1 & 1 \\
     0 & \latticevelocity & -\latticevelocity \\
     0 & \latticevelocity^2/2 & \latticevelocity^2/2 \\
    \end{matrix}
    \right ), \quad 
    \operatorial{S} = \text{diag}(0, 0, p), \quad \momentum^{2, \text{eq}} = \frac{c^2}{2}u.
\end{equation*}
We consider $\latticevelocity = 1$, $p = 1.7$ and a final time $T=1.5625$ for each simulation.
The initial data are taken at equilibrium.

\subsection{Results}

To quantify the amplitude of the reflected wave, we simulate on different configurations:
\begin{itemize}
    \item $\average{u}^{\text{jump}}$ is the solution obtained using the spatial discretization illustrated on Fig.~ \ref{fig:configuration}, the left sub-domain $\Omega_{\text{left}}$ is finely meshed with level $\levelleft = \maxlevel$, whereas the right sub-domain $\Omega_{\text{right}}$ is coarsely meshed with level $\levelright = \minlevel \leq \levelleft$. We vary the jump gap $\leveljump := |\levelleft - \levelright|$. This configuration is the monodimensional equivalent of that presented by \cite{lagrava2012revisiting}.
    \item $\average{u}^{\text{ref}}$ is the reference solution obtained on the uniform mesh at the finest level $\maxlevel$.
    \item $\average{u}^{\text{coarse}}$ is the solution obtained on the uniform mesh at the coarsest level $\minlevel$.
\end{itemize}

We introduce the $L^1$-norm of the exact solution $u$ for the sake of normalizing: $\lVert u(t, \cdot)\rVert_{1, \spacestep} = \sum_{\indexletter} \spacestep |u(t, x_{\maxlevel, \indexletter}) | \simeq \lVert u_0\rVert_{1, \spacestep}$. We measure
\begin{align*}
\text{E}_{\text{ref}}(t) &= \dfrac{\lVert \average{u}^{\text{ref}}(t, \cdot) - u(t, \cdot) \rVert_{1, \spacestep} }{\lVert u(t, \cdot)\rVert_{1, \spacestep}},  \\
 \text{E}_{\text{coarse}}(t) &= \dfrac{\lVert \reconstructed{\average{u}}{\text{coarse}}{}(t) - u(t, \cdot) \rVert_{1, \spacestep} }{\lVert u(t, \cdot)\rVert_{1, \spacestep}}, \qquad \text{D}_{\text{coarse}}(t) = \dfrac{\lVert \reconstructed{\average{u}}{\text{coarse}}{}(t) - \average{u}^{\text{ref}}_{}(t) \rVert_{1, \spacestep} }{\lVert u(t, \cdot)\rVert_{1, \spacestep}}, 
\end{align*}
\begin{align*}
 \text{E}_{\text{jump}}(t) &= \dfrac{\lVert \reconstructed{\average{u}}{\text{jump}}{}(t) - u(t,\cdot) \rVert_{1, \spacestep} }{\lVert u(t, \cdot)\rVert_{1, \spacestep}}, \qquad \text{D}_{\text{jump}}(t) = \dfrac{ \lVert \reconstructed{\average{u}}{\text{jump}}{}(t) - \average{u}^{\text{ref}}_{}(t)\rVert_{1, \spacestep} }{\lVert u(t, \cdot)\rVert_{1, \spacestep}},\\
 \text{D}_{\text{jump-refl}}(t) &= \dfrac{\lVert ( {\average{u}}^{\text{jump}}_{}(t) - \average{u}^{\text{ref}}_{}(t) ) \chi_{\{\indexletter ~:~ \cellletter_{\levelleft, \indexletter} \subset \Omega_{\text{left}}\} } \rVert_{1, \spacestep} }{\lVert u(t, \cdot)\rVert_{1, \spacestep}}.
\end{align*}

Here the double hat operator $\reconstructed{~}{~}{~}$ represents the application of the same interpolation used to update the ghost cells. This step is needed to compare solution on the same mesh.
Using the analysis of \cite{bellotti2021lbmmreqeq}, we can prove an estimate on the amplitude of the reflected wave $\text{D}_{\text{jump-refl}}(T)$:
\begin{proposition}\label{prop:Convergence}
    Assume the solution of the problem with level jump in the computational mesh to be smooth, namely of class $\mathcal{C}^4$, for any time $t \in [0, T]$, then we have
    \begin{equation*}
        \text{D}_{\mathrm{jump-refl}}(T) = \bigO{\spacestep^4}.
    \end{equation*}
\end{proposition}
\begin{proof}
    Let us introduce the following indices for the cells close to the level jump
    \begin{align*}
            \overleftarrow{\indexletter} &= \{ \indexletter \quad \text{such that} \quad \cellletter_{\levelleft, \indexletter} \subset \Omega_{\text{left}} \quad \text{and} \quad \cellletter_{\levelleft, \indexletter}  \cap \partial \Omega_{\text{left}} \neq \varnothing \}, \\
            \overrightarrow{\indexletter} &= \{ \indexletter \quad \text{such that} \quad \cellletter_{\levelleft, \indexletter} \subset \Omega_{\text{right}} \quad \text{and} \quad \cellletter_{\levelleft, \indexletter}  \cap \partial \Omega_{\text{right}} \neq \varnothing \}, \\
            \underrightarrow{\indexletter} &= \{ \indexletter \quad \text{such that} \quad \cellletter_{\levelright, \indexletter} \subset \Omega_{\text{right}} \quad \text{and} \quad \cellletter_{\levelright, \indexletter}  \cap \partial \Omega_{\text{right}} \neq \varnothing \},
    \end{align*}
    then performing the Taylor expansions detailed in \cite{bellotti2021lbmmreqeq}, we claim that for every time $t$
    \begin{equation*}
        \average{\distribution}_{\levelright, \underrightarrow{\indexletter} }^{\text{jump}, \populationindex}(t) = \average{\distribution}_{\levelleft, \overrightarrow{\indexletter} }^{\text{ref}, \populationindex}(t) + \bigO{\spacestep^4}.
    \end{equation*}
    The collision phase is linear, otherwise we may use a local Lipschitz property of the equilibria. 
    \begin{equation*}
        \average{\distribution}_{\levelright, \underrightarrow{\indexletter} }^{\text{jump}, \populationindex, \collided}(t) = \average{\distribution}_{\levelleft, \overrightarrow{\indexletter} }^{\text{ref}, \populationindex, \collided}(t) + \bigO{\spacestep^4}.
    \end{equation*}
    The update of the ghost cell $\cellletter_{\levelleft, \overrightarrow{\indexletter}}$ is a linear combination of values between which we have those on $\cellletter_{\levelright, \underrightarrow{\indexletter}}$, therefore
    \begin{equation*}
        \average{\distribution}_{\levelleft, \overrightarrow{\indexletter} }^{\text{jump}, \populationindex, \collided}(t) = \average{\distribution}_{\levelleft, \overrightarrow{\indexletter} }^{\text{ref}, \populationindex, \collided}(t) + \bigO{\spacestep^4}.
    \end{equation*}
    Concentrating on the stream phase of $\normalizedvelocityletter_2 = -1$, we have
    \begin{equation*}
        \begin{cases}
            \average{\distribution}_{\levelleft, \overleftarrow{\indexletter}}^{\text{jump}, 2}(t + \timestep) = \average{\distribution}_{\levelleft, \overrightarrow{\indexletter} }^{\text{jump}, 2, \collided}(t) &= \average{\distribution}_{\levelleft, \overrightarrow{\indexletter} }^{\text{ref}, 2, \collided}(t) + \bigO{\spacestep^4}, \\
            \average{\distribution}_{\levelleft, \overleftarrow{\indexletter}}^{\text{ref}, 2}(t + \timestep) &= \average{\distribution}_{\levelleft, \overrightarrow{\indexletter} }^{\text{ref}, 2, \collided}(t).
        \end{cases}
    \end{equation*}
    Thus by computing the moment $u$ and taking the difference, we deduce that for every time
    \begin{equation*}
        |\average{u}_{\levelleft, \overleftarrow{\indexletter}}^{\text{jump}}(t + \timestep) - \average{u}_{\levelleft, \overleftarrow{\indexletter}}^{\text{ref}}(t + \timestep)| = \bigO{\spacestep^4}.
    \end{equation*}
    The CFL condition on the finest resolution imposes that information, thus the errors, propagate of one cell for one time step. The constant $C > 0$ carries the normalization
    \begin{align*}
        \text{D}_{\text{jump-refl}}(T) &= C \Bigg ( \sum_{\substack{\indexletter \neq \overleftarrow{\indexletter} \text{s.t.} \\ \cellletter_{\maxlevel, \indexletter} \subset \Omega_{\text{left}}}} \hspace{-0.5cm} \spacestep |\average{u}^{\text{jump}}_{\maxlevel, \indexletter}(T) - \average{u}^{\text{ref}}_{\maxlevel, \indexletter}(T)|  + \spacestep |\average{u}^{\text{jump}}_{\maxlevel, \overleftarrow{\indexletter} }(T) - \average{u}^{\text{ref}}_{\maxlevel, \overleftarrow{\indexletter} }(T)| \Bigg ), \\
        &\leq C \adaptiveroundbrackets{\dfrac{1}{C} \text{D}_{\text{jump-refl}}(T - \timestep) + \bigO{\spacestep^5}} \leq \dots \leq \dfrac{T}{\timestep} \bigO{\spacestep^5} = \dfrac{\latticevelocity T}{\spacestep} \bigO{\spacestep^5}, \\
        &= \bigO{\spacestep^4}.
    \end{align*}

\end{proof}

The results of the numerical simulations are provided in Table \ref{tab:gaussian_fine_coarse}.
We have also run comparisons with the method from Fakhari and Lee \cite{fakhari2014finite} and that of Rohde \emph{et al.} \cite{rohde2006generic} presented in Fig.~ \ref{fig:comparison}.

    \begin{table}\caption{\label{tab:gaussian_fine_coarse}Results for the transition between fine and coarse mesh with the Gaussian as initial datum. Numerical convergence rates are reported between parenthesis.}
        \begin{center}
            \begin{footnotesize}
            \begin{tabular}{cccccccc}
            $\maxlevel$ & $\textrm{E}_{\text{ref}}(T)$ & $\textrm{E}_{\text{coarse}}(T)$ & $\textrm{D}_{\text{coarse}}(T)$ & $\textrm{E}_{\text{jump}}(T)$ & $\textrm{D}_{\text{jump}}(T)$ & $\textrm{D}_{\text{jump-refl}}(T)$\\
            \toprule
             & \multicolumn{6}{c}{$\levelletter_{\text{jump}} = 1$} \\
            \midrule
7	& 7.30E-02	\phantom{(0.95)}  & 8.19E-02	\phantom{(0.95)}  & 1.30E-02	\phantom{(0.95)}  & 7.49E-02	\phantom{(0.95)}  & 2.64E-03	\phantom{(0.95)}  & 1.27E-05	\phantom{(0.95)} \\   
8	& 3.78E-02	(0.95)	& 3.89E-02	(1.07)	& 1.86E-03	(2.81)	& 3.80E-02	(0.98)	& 3.84E-04	(2.78)	& 8.90E-07	(3.83) \\
9	& 1.92E-02	(0.98)	& 1.93E-02	(1.01)	& 2.28E-04	(3.03)	& 1.93E-02	(0.98)	& 5.19E-05	(2.89)	& 5.95E-08	(3.90) \\
10	& 9.70E-03	(0.99)	& 9.70E-03	(0.99)	& 2.49E-05	(3.20)	& 9.71E-03	(0.99)	& 6.75E-06	(2.94)	& 3.86E-09	(3.95) \\
11	& 4.87E-03	(0.99)	& 4.87E-03	(0.99)	& 3.67E-06	(2.76)	& 4.87E-03	(0.99)	& 8.65E-07	(2.96)	& 2.46E-10	(3.97) \\
12	& 2.44E-03	(1.00)	& 2.44E-03	(1.00)	& 1.08E-06	(1.77)	& 2.44E-03	(1.00)	& 1.11E-07	(2.96)	& 1.55E-11	(3.99) \\
13	& 1.22E-03	(1.00)	& 1.22E-03	(1.00)	& 3.11E-07	(1.79)	& 1.22E-03	(1.00)	& 1.44E-08	(2.95)	& 9.76E-13	(3.99) \\
            \midrule
             & \multicolumn{6}{c}{$\levelletter_{\text{jump}} = 2$} \\
            \midrule
7	& 7.30E-02	\phantom{(0.95)}  & 1.61E-01	\phantom{(0.95)}	    & 9.47E-02	\phantom{(0.95)}	    & 9.22E-02	\phantom{(0.95)}	    & 2.21E-02	\phantom{(0.95)}	    & 4.84E-04 \phantom{(0.95)} \\
8	& 3.78E-02	(0.95)	& 5.00E-02	(1.68)	& 1.68E-02	(2.50)	& 4.04E-02	(1.19)	& 3.50E-03	(2.66)	& 3.10E-05	(3.96) \\
9	& 1.92E-02	(0.97)	& 2.06E-02	(1.28)	& 2.24E-03	(2.90)	& 1.96E-02	(1.05)	& 4.71E-04	(2.89)	& 2.03E-06	(3.94) \\
10	& 9.70E-03	(0.99)	& 9.82E-03	(1.07)	& 2.63E-04	(3.09)	& 9.74E-03	(1.01)	& 6.07E-05	(2.96)	& 1.31E-07	(3.96) \\
11	& 4.87E-03	(0.99)	& 4.87E-03	(1.01)	& 2.78E-05	(3.24)	& 4.88E-03	(1.00)	& 7.70E-06	(2.98)	& 8.32E-09	(3.97) \\
12	& 2.44E-03	(1.00)	& 2.44E-03	(1.00)	& 4.67E-06	(2.57)	& 2.44E-03	(1.00)	& 9.70E-07	(2.99)	& 5.25E-10	(3.99) \\
13	& 1.22E-03	(1.00)	& 1.22E-03	(1.00)	& 1.39E-06	(1.75)	& 1.22E-03	(1.00)	& 1.22E-07	(2.99)	& 3.30E-11	(3.99) \\
            \midrule
             & \multicolumn{6}{c}{$\levelletter_{\text{jump}} = 3$} \\
                \midrule
7	& 7.30E-02	\phantom{(0.95)} & 4.33E-01	\phantom{(0.95)}   & 3.62E-01	\phantom{(0.95)}   & 1.83E-01	\phantom{(0.95)}   & 1.12E-01	\phantom{(0.95)}   & 9.81E-03	\phantom{(0.95)} \\
8	& 3.78E-02	(0.95)	& 1.25E-01	(1.79)	& 9.31E-02	(1.96)	& 5.79E-02	(1.66)	& 2.26E-02	(2.31)	& 7.10E-04	(3.79) \\
9	& 1.92E-02	(0.97)	& 3.02E-02	(2.05)	& 1.45E-02	(2.68)	& 2.16E-02	(1.42)	& 3.12E-03	(2.86)	& 4.13E-05	(4.10) \\
10	& 9.70E-03	(0.99)	& 1.07E-02	(1.49)	& 1.80E-03	(3.01)	& 9.98E-03	(1.12)	& 3.97E-04	(2.98)	& 2.64E-06	(3.97) \\
11	& 4.87E-03	(0.99)	& 4.94E-03	(1.12)	& 1.99E-04	(3.18)	& 4.91E-03	(1.02)	& 4.96E-05	(3.00)	& 1.68E-07	(3.97) \\
12	& 2.44E-03	(1.00)	& 2.44E-03	(1.02)	& 2.13E-05	(3.23)	& 2.45E-03	(1.00)	& 6.19E-06	(3.00)	& 1.06E-08	(3.99) \\
13	& 1.22E-03	(1.00)	& 1.22E-03	(1.00)	& 5.23E-06	(2.03)	& 1.22E-03	(1.00)	& 7.74E-07	(3.00)	& 6.66E-10	(3.99) \\
                        \midrule
             & \multicolumn{6}{c}{$\levelletter_{\text{jump}} = 4$} \\
                \midrule
7	& 7.30E-02	\phantom{(0.95)} & 9.10E-01	\phantom{(0.95)}   & 8.43E-01	\phantom{(0.95)}   & 4.00E-01	\phantom{(0.95)}   & 3.29E-01	\phantom{(0.95)}  & 6.02E-02	\phantom{(0.95)} \\
8	& 3.78E-02	(0.95)	& 4.24E-01	(1.10)	& 3.88E-01	(1.12)	& 1.57E-01	(1.35)	& 1.20E-01	(1.45)	& 1.43E-02	(2.07) \\
9	& 1.92E-02	(0.97)	& 9.69E-02	(2.13)	& 8.24E-02	(2.24)	& 3.74E-02	(2.07)	& 2.02E-02	(2.58)	& 8.21E-04	(4.13) \\
10	& 9.70E-03	(0.99)	& 1.78E-02	(2.44)	& 1.08E-02	(2.93)	& 1.16E-02	(1.69)	& 2.42E-03	(3.06)	& 4.42E-05	(4.22) \\
11	& 4.87E-03	(0.99)	& 5.47E-03	(1.71)	& 1.21E-03	(3.15)	& 5.08E-03	(1.20)	& 2.89E-04	(3.07)	& 2.85E-06	(3.95) \\
12	& 2.44E-03	(1.00)	& 2.45E-03	(1.16)	& 1.25E-04	(3.27)	& 2.47E-03	(1.04)	& 3.51E-05	(3.04)	& 1.83E-07	(3.97) \\
13	& 1.22E-03	(1.00)	& 1.21E-03	(1.02)	& 1.98E-05	(2.66)	& 1.23E-03	(1.01)	& 4.34E-06	(3.02)	& 1.15E-08	(3.99) \\		
                         \midrule
             & \multicolumn{6}{c}{$\levelletter_{\text{jump}} = 5$} \\
             \midrule            
7	& 7.30E-02	\phantom{(0.95)} & 1.24E+00	\phantom{(0.95)}   & 1.20E+00	\phantom{(0.95)}   & 6.46E-01	\phantom{(0.95)}   & 5.82E-01	\phantom{(0.95)}   & 1.22E-01	\phantom{(0.95)} \\
8	& 3.78E-02	(0.95)	& 9.30E-01	(0.41)	& 8.96E-01	(0.42)	& 3.88E-01	(0.74)	& 3.50E-01	(0.73)	& 7.35E-02	(0.73) \\
9	& 1.92E-02	(0.97)	& 4.51E-01	(1.04)	& 4.34E-01	(1.05)	& 1.45E-01	(1.42)	& 1.27E-01	(1.47)	& 1.86E-02	(1.98) \\
10	& 9.70E-03	(0.99)	& 8.05E-02	(2.49)	& 7.44E-02	(2.54)	& 2.58E-02	(2.49)	& 1.78E-02	(2.83)	& 8.87E-04	(4.39) \\
11	& 4.87E-03	(0.99)	& 1.06E-02	(2.93)	& 7.74E-03	(3.27)	& 6.32E-03	(2.03)	& 1.80E-03	(3.31)	& 4.31E-05	(4.36) \\
12	& 2.44E-03	(1.00)	& 2.68E-03	(1.98)	& 7.52E-04	(3.36)	& 2.58E-03	(1.29)	& 2.00E-04	(3.18)	& 2.84E-06	(3.93) \\
13	& 1.22E-03	(1.00)	& 1.18E-03	(1.18)	& 8.06E-05	(3.22)	& 1.24E-03	(1.06)	& 2.35E-05	(3.08)	& 1.86E-07	(3.93) \\
                \hline
            \end{tabular}
            \end{footnotesize}
        \end{center}
    \end{table}
    
    \begin{figure}
       \begin{center}
            \includegraphics[width=1.\textwidth]{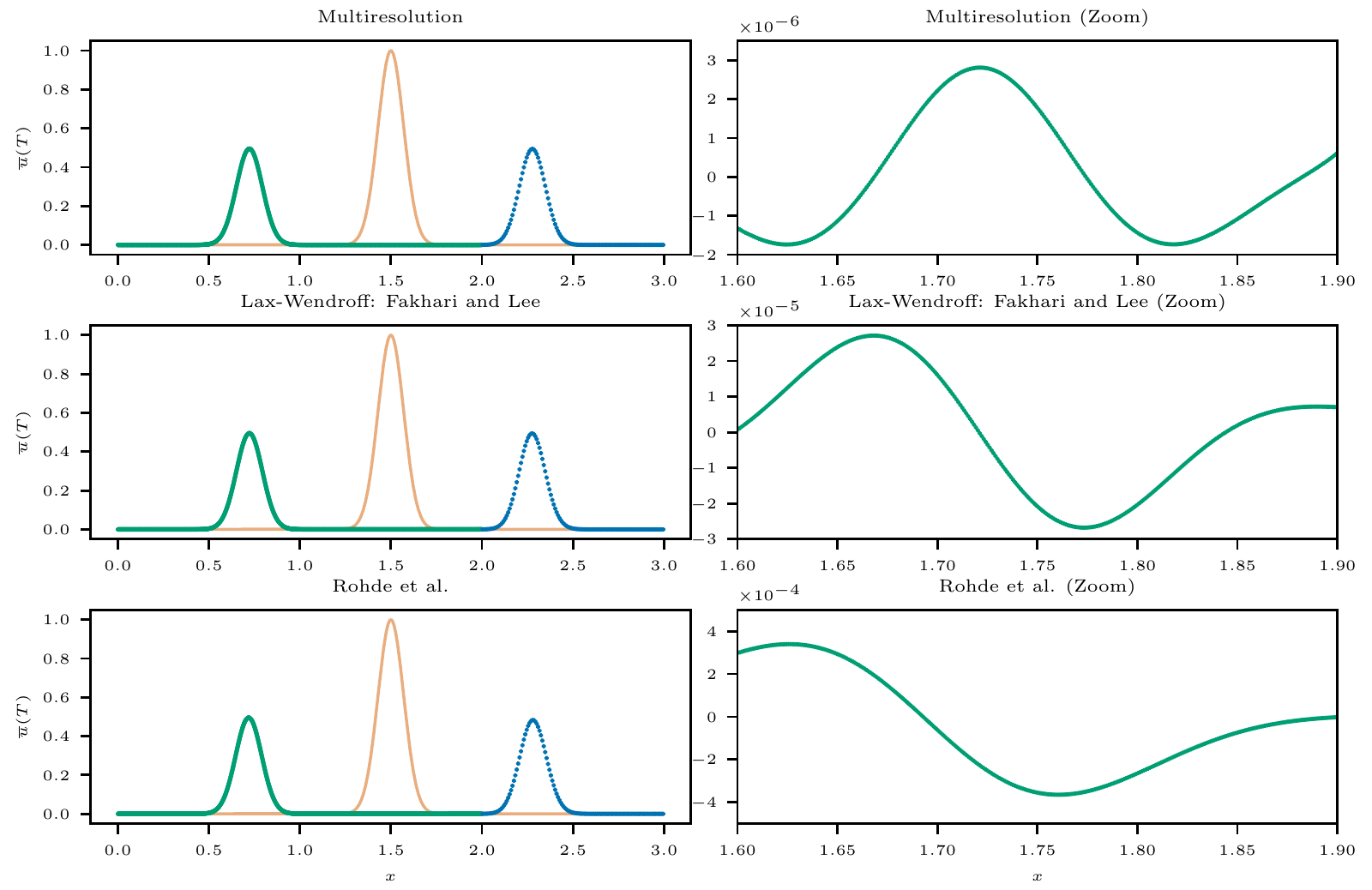}
        \end{center}\caption{\label{fig:comparison}Results of the simulation on the mesh with jump (whole domain on the left, magnification on $[1.6, 1.9]$ and on the $y$ axis on the right). Initial solution in pale orange and solution at $t = T$ in green (left subdomain) and blue (right subdomain). On the first row, we use our multiresolution scheme \cite{bellotti2021multiresolution1d, bellotti2021multidimensional2d, bellotti2021lbmmreqeq}. On the second row, we use the Lax-Wendroff scheme by Fakhari and Lee \cite{fakhari2014finite}. On the third row, the scheme with local time-stepping by Rohde \emph{et al.} \cite{rohde2006generic}. The simulation uses $\levelleft = 10$ and $\leveljump = 3$.}
    \end{figure} 


\subsection{Discussion}

Commenting on Table \ref{tab:gaussian_fine_coarse}, we see that the reference scheme converges linearly $\text{E}_{\text{ref}}(T) = \bigO{\spacestep}$ once refining as expected from the analysis by the equivalent equations \cite{dubois2008equivalent}. The error $\text{E}_{\text{coarse}}(T) \leq \text{E}_{\text{ref}}(T) + \text{D}_{\text{coarse}}(T) = \bigO{\spacestep} + \bigO{\spacestep^3} = \bigO{\spacestep}$ converges linearly as well because the additional difference $\text{D}_{\text{coarse}}(T) = \bigO{\spacestep^3}$ does not influence the overall convergence, as pointed out in \cite{bellotti2021lbmmreqeq}.
The same behavior is observed for the mesh with jump, namely for $\text{E}_{\text{jump}}(T)$ and $\text{D}_{\text{jump}}(T)$. Very interestingly $\text{D}_{\text{jump}}(T) \leq \text{D}_{\text{coarse}}(T)$: counter-intuitively this is \emph{a priori} not granted due to the possible formation of waves reflected at the jump, even though only a part of the domain is coarsened.
This gives a first indication about the fact that the reflected waves are perfectly mastered.
The second indication comes from $\text{D}_{\text{jump-refl}}(T) = \bigO{\spacestep^4}$.
This means that with our method, we are able to decrease the amplitude of the reflected waves with fourth-order convergence in the space step, in accordance with Prop.~\ref{prop:Convergence}.
The supra-convergence compared to $\text{D}_{\text{jump}}(T)$ comes from the fact that at each time step, the reflected wave is generated only on the cell of $\Omega_{\text{right}}$ next to the interface, so that it eventually propagates to the left inside the fine medium without additional amplification of the error.
Observe that the convergence rates worsen for large $\leveldifference$ and for small $\maxlevel$ due to the fact that we are no longer allowed to perform the Taylor expansions needed by Proposition \ref{prop:Convergence}, which are done at the current level of resolution $\levelletter$. Indeed, in this case, one can no longer claim that $2^{\leveldifference} \spacestep$ is $\bigO{\spacestep}$.

Compared to other methods, we can show with the same proof path than Prop.~\ref{prop:Convergence} that the Lax-Wendroff strategy by \cite{fakhari2014finite} yields $\text{D}_{\text{jump-refl}}(T) = \bigO{\spacestep^3}$, which is one order less than our method. This can also be qualitatively seen on Fig.~ \ref{fig:comparison}.
Concerning the approach by \cite{rohde2006generic} where local time-stepping is used, we observe that it yields quite large reflected waves.
This waves are one order of magnitude larger than for the method by \cite{fakhari2014finite} and two orders of magnitude larger than our approach.
However, the local time-stepping prevents us from applying the same theoretical study to this scheme.

Mastering reflected waves at a high order of accuracy is important when our technique is extended to typical multidimensional applications. When simulating the incompressible Navier-Stokes equations \emph{via} a quasi-incompressible D2Q9 scheme, spurious acoustic waves are of order $\bigO{\spacestep^2}$, thus controlling their reflection at order $\bigO{\spacestep^4}$ is a highly desirable feature of the scheme.



\section{Conclusions}

In this contribution, we have briefly presented our adaptive lattice Boltzmann method, which is studied in Prop.~\ref{prop:Convergence} using the technique introduced in \cite{bellotti2021lbmmreqeq} to conclude that in case of a fixed mesh jump, the amplitude of the spuriously reflected waves is of order $\bigO{\spacestep^4}$.
This fact is numerically verified and compared to the performance of other approaches available in the literature \cite{fakhari2014finite} and \cite{rohde2006generic}, showing that our method outperforms these traditional approaches. 

It is worthwhile observing that the original method \cite{bellotti2021multiresolution1d, bellotti2021multidimensional2d} was conceived to be used with dynamically adapted meshes which automatically follow waves and fronts with finer discretizations once their lack of regularity justifies the depart from a coarse uniform mesh.
Thus, in this case, we even do not expect the $\bigO{\spacestep^4}$ perturbation because fronts never cross level jumps but are precisely and successfully ``chased'' by the fine discretization. 
We observe that multiresolution is not relevant for systems developing homogeneous isotropic turbulence. However, dealing with spatially large problems where turbulent flows at high Reynolds number are present in a small portion of the domain could be interesting and advantageous. The analysis of this framework could be tackled in publications to come with a tailored data structure.

    
For the sake of a quick and effective presentation, we have restrained the study to the one-dimensional setting. However, the generalization to higher spatial dimensions is straightforward and follows the indications of \cite{bellotti2021lbmmreqeq}, since the operators involved in the multiresolution analysis are extended \cite{bellotti2021multidimensional2d} by tensor product in the remaining directions of the space.
We have proved in \cite{bellotti2021lbmmreqeq} that the multidimensional extension of the method retains the same accuracy levels in every spatial direction.

\section*{Acknowledgments}

The authors thank Pierre Lallemand and Fran\c{c}ois Dubois for the stimulating discussions during the ``groupe de travail Schémas de Boltzmann sur réseau'' at the IHP Paris and the anonymous referee for the valuable inputs and enhancements he/she suggested.

\bibliographystyle{acm}
\bibliography{biblio}

\end{document}